\documentclass[12pt]{amsart}

\usepackage{amsfonts,amsmath,amsthm,amssymb,latexsym,amsthm,newlfont,enumerate,color}

\newif\ifslide

\theoremstyle{plain}

\ifdefined\spacer
\newtheorem{theorem}{Theorem}
 \else
\newtheorem{theorem}{Theorem}[section]
\fi

\newtheorem{theorema}{Theorem}

\newtheorem{corollary}[theorem]{Corollary}
\newtheorem{lemma}[theorem]{Lemma}

\newtheorem{proposition}[theorem]{Proposition}
\newtheorem{definition-lemma}[theorem]{Definition-Lemma}
\newtheorem{question}[theorem]{Question}
\newtheorem{red-question}[theorem]{\textcolor{red}{Question}}

\newtheorem{conjecture}[theorem]{Conjecture}

\theoremstyle{definition}
\newtheorem{definition}[theorem]{Definition}
\newtheorem{remark}[theorem]{Remark}

\newtheorem{example}[theorem]{Example}

\def\ideal#1.{I_{#1}}
\def\ring#1.{\mathcal {O}_{#1}}

\def\ddiv{\operatorname {div}}
\def\base{\operatorname {Bs}}

\def\fring#1.{\hat{\mathcal {O}}_{#1}}
\def\proj#1.{\mathbb {P}(#1)}
\def\pr #1.{\mathbb {P}^{#1}}
\def\dpr #1.{\hat{\mathbb {P}}^{#1}}
\def\af #1.{\mathbb A^{#1}}
\def\Hz #1.{\mathbb F_{#1}}
\def\Hbz #1.{\overline{\mathbb F}_{#1}}
\def\fb#1.{\underset #1 {\times}}
\def\rest#1.{\underset {\ \ring #1.} \to \otimes}
\def\au#1.{\operatorname {Aut}\,(#1)}
\def\deg#1.{\operatorname {deg } (#1)}

\def\pic#1.{\operatorname {Pic}\,(#1)}
\def\pico#1.{\operatorname{Pic}^0(#1)}
\def\picg#1.{\operatorname {Pic}^G(#1)}
\def\ner#1.{NS (#1)}
\def\rdown#1.{\llcorner#1\lrcorner}
\def\rfdown#1.{\lfloor{#1}\rfloor}
\def\rup#1.{\ulcorner{#1}\urcorner}
\def\rcup#1.{\lceil{#1}\rceil}

\def\n1#1.{\operatorname {N_1}(#1)}  
\def\cn1#1.{\overline{\operatorname {N^1}(#1)}} 
\def\cone#1.{\operatorname {NE}(#1)}     
\def\ccone#1.{\overline{\operatorname {NE}}(#1)}
\def\none#1.{\operatorname {NF}(#1)}
\def\cnone#1.{\overline{\operatorname {NF}}(#1)}
\def\mone#1.{\operatorname {NM}(#1)} 
\def\cmone#1.{\overline{\operatorname {NM}}(#1)}

\def\coef#1.{\frac{(#1-1)}{#1}}
\def\vit#1.{D_{\langle #1 \rangle}}
\def\mm#1.{\overline {M}_{0,#1}}
\def\H1#1.{H^1(#1,{\ring #1.})}
\def\ac#1.{\overline {\mathbb F}_{#1}}

\def\adj#1.{\frac {#1-1}{#1}}
\def\spn#1.{\overline{#1}}
\def\pek#1.#2.{\Cal P^{#1}(#2)}
\def\plk#1.#2.{\Cal P^{\leq #1}(#2)}
\def\ev#1.{\operatorname{ev_{#1}}}
\def\ilist#1.{{#1}_1,{#1}_2,\dots}
\def\bminv#1.{(\nu_1,s_1;\nu_2,s_2;\dots ;\nu_{#1},s_{#1};\nu_{r+1})}
\def\zinv#1.{(\nu_1,s_1;\nu_2,s_2;\dots ;\nu_{#1},s_{#1};0)}
\def\iinv#1.{(\nu_1,s_1;\nu_2,s_2;\dots ;\nu_{#1},s_{#1};\infty)}

\def\scr #1.{\mathcal #1}

\def\llist#1.#2.{{#1}_1,{#1}_2,\dots,{#1}_{#2}}
\def\ulist#1.#2.{{#1}^1,{#1}^2,\dots,{#1}^{#2}}
\def\lomitlist#1.#2.{{#1}_1,{#1}_2,\dots,\hat {{#1}_i}, \dots, {#1}_{#2}}
\def\lomitlistz#1.#2.{{#1}_0,{#1}_1,\dots,\hat {{#1}_i}, \dots, {#1}_{#2}}
\def\loc#1.#2.{\Cal O_{#1,#2}}
\def\fderiv#1.#2.{\frac {\partial #1}{\partial #2}}
\def\deriv#1.#2.{\frac {d #1}{d #2}}

\def\map#1.#2.{#1 \longrightarrow #2}
\def\rmap#1.#2.{#1 \dasharrow #2}
\def\emb#1.#2.{#1 \hookrightarrow #2}
\def\non#1.#2.{\text {Spec }#1[\epsilon]/(\epsilon)^{#2}}
\def\Hi#1.#2.{\text {Hilb}^{#1}(#2)}
\def\sym#1.#2.{\operatorname {Sym}^{#1}(#2)}
\def\Hb#1.#2.{\text {Hilb}_{#1}(#2)}
\def\Hm#1.#2.{\Hom_{#1}(#2)}
\def\prd#1.#2.{{#1}_1\cdot {#1}_2\cdots {#1}_{#2}}
\def\Bl #1.#2.{\operatorname {Bl}_{#1}#2}
\def\pl #1.#2.{#1^{\otimes #2}}
\def\mgn#1.#2.{\overline {M}_{#1,#2}}
\def\ialist#1.#2.{{#1}_1 #2 {#1}_2, #2\dots}
\def\pair#1.#2.{\langle #1, #2\rangle}
\def\vandermonde#1.#2.{\left|
\begin{matrix}
1 & 1 & 1 & \dots & 1\\
{#1}_1 & {#1}_2 & {#1}_3 & \dots & {#1}_{#2}\\
{#1}_1^2 & {#1}_2^2 & {#1}_3^2 & \dots & {#1}_{#2}^2\\
\vdots & \vdots & \vdots & \ddots & \vdots\\
{#1}_1^{#2-1} & {#1}_2^{#2-1} & {#1}_2^{#2-1} & \dots & {#1}_{#2}^{#2-1}\\
\end{matrix}
\right|
}
\def\vandermondet#1.#2.{\left|
\begin{matrix}
1 & {#1}_1   & {#1}_1^2 & \dots & {#1}_1^{#2-1}\\
1 & {#1}_2   & {#1}_2^2 & \dots & {#1}_2^{#2-1}\\
1 & {#1}_3   & {#1}_3^2 & \dots & {#1}_3^{#2-1}\\
\vdots & \vdots & \vdots & \ddots & \vdots\\
1 & {#1}_{#2}& {#1}_{#2}^2 & \dots & {#1}_{#2}^{#2-1}\\
\end{matrix}
\right|
}
\def\gr#1.#2.{\mathbb{G}(#1,#2)}

\def\alist#1.#2.#3.{{#1}_1 #2 {#1}_2 #2\dots #2 {#1}_{#3}}
\def\zlist#1.#2.#3.{#1_0 #2 #1_1 #2\dots #2 #1_{#3}}
\def\lomitlist30#1.#2.#3.{{#1}_0,{#1}_1 #2 \dots #2\hat {{#1}_i} #2\dots #2 {#1}_{#3}}
\def\lmap#1.#2.#3.{#1 \overset{#2}{\longrightarrow} #3}
\def\mes#1.#2.#3.{#1 \longrightarrow #2 \longrightarrow #3}
\def\ses#1.#2.#3.{0\longrightarrow #1 \longrightarrow #2 \longrightarrow #3 \longrightarrow 0}
\def\les#1.#2.#3.{0\longrightarrow #1 \longrightarrow #2 \longrightarrow #3}
\def\res#1.#2.#3.{#1 \longrightarrow #2 \longrightarrow #3\longrightarrow 0}
\def\Hi#1.#2.#3.{\text {Hilb}^{#1}_{#2}(#3)}
\def\ten#1.#2.#3.{#1\underset {#2}{\otimes} #3}
\def\lomitlist30#1.#2.#3.{{#1}_0 #2 {#1}_1 #2 \dots #2 \hat {{#1}_i} #2 \dots #2 {#1}_{#3}}
\def\mderiv#1.#2.#3.{\frac {d^{#3} #1}{d #2^{#3}}}

\def\Hom{\operatorname{Hom}}

\def\Supp{\operatorname{Supp}}
\def\Bs{\operatorname{\mathbf B}}
\def\Exc{\operatorname{Exc}}

\def\dim{\operatorname{dim}}

\def\deg{\operatorname{deg}}

\def\det{\operatorname{det}}

\def\Div{\operatorname{Div}}

\def\mult{\operatorname{mult}}

\def\mob{\operatorname{Mob}}
\def\fix{\operatorname{Fix}}

\def\rest{\operatorname{res}}

\def\bs{\operatorname{Bs}}

\def\C{\mathbb C}

\def\e{\Cal E}

\def\e1{E_1}
\def\e2{E_2}

\def\mapdown#1{\big\downarrow\rlap{$\vcenter{\hbox{$\scriptstyle#1$}}$}}

\def\mapse#1{
{\vcenter{\hbox{$\mathop{\smash{\raise1pt\hbox{$\diagdown$}\!\lower7pt
\hbox{$\searrow$}}\vphantom{p}}\limits_{#1}\vphantom{\mapdown{}}$}}}}

\def\VR#1.{height#1pt&\omit&&\omit&&\omit&&\omit&&\omit&\cr}

\def\VRT#1.{height#1pt&\omit&&\omit&\cr}

\usepackage{graphicx}
\usepackage{hyperref}
\usepackage{epsfig}
\begin{document}

\title{The Minimal Model Program Revisited}

\author{Paolo Cascini}
\address{Department of Mathematics\\
Imperial College London\\
180 Queen's Gate\\
London SW7 2AZ, UK}
\email{p.cascini@imperial.ac.uk}
\author{Vladimir Lazi\'c}
\address{Mathematisches Institut\\
Universit\"at Bayreuth\\
95440 Bayreuth\\
Germany}
\email{vladimir.lazic@uni-bayreuth.de}

\thanks{The first author was partially supported by an EPSRC grant. We would like to thank the referee for many useful comments.}

\begin{abstract}
We give a light introduction to some recent developments in Mori theory, and  to our recent direct proof of the finite generation of the canonical ring.
\end{abstract}

\maketitle
\tableofcontents

\section{Introduction}

The purpose of Mori theory is to give a meaningful birational classification of a large class of algebraic varieties. This means several things: that varieties have mild singularities (in the sense explained below), that we do some surgery operations on varieties which do not make singularities worse, and that objects which are end-products of surgery have some favourable properties. Smooth projective varieties belong to this class, and even though the end-product of doing Mori surgery on a smooth variety might not be -- and usually is not -- smooth, it is nevertheless good enough for all practical purposes.

One of the main properties of these surgeries is that they do not change the canonical ring. Recall that given a smooth projective variety $X$, the {\em canonical ring\/} is
$$R(X,K_X)=\bigoplus_{m\ge 0} H^0(X,\ring X. (mK_X)).$$
More generally, for any $\mathbb Q$-divisors $D_1,\dots,D_k$ on $X$, we can consider the {\em divisorial ring\/}
$$R(X;D_1,\dots,D_k)=\bigoplus_{(m_1,\dots,m_k)\in\mathbb{N}^k} H^0\big(X,\ring X. \big(\rfdown\textstyle\sum m_iD_i.\big)\big).$$
Rings of this type were extensively studied in Zariski's seminal paper \cite{Zar62}. In particular, the following is implicit in \cite{Zar62}:
\begin{question}\label{q_finite}
Let $X$ be a smooth projective variety.
Is the canonical ring $R(X,K_X)$ finitely generated?
\end{question}
The answer is trivially affirmative for curves, and in the appendix of \cite{Zar62}, Mumford confirmed it when $X$ is a surface. Since it is invariant under operations of Mori theory, this question largely motivated conjectures of Mori theory, which postulate that Mori program should terminate with a variety with $K_X$ semiample.
This was confirmed in the case of threefolds in the 1980s by Mori et al. Very recently, the question was answered affirmatively in any dimension in
\cite{BCHM10} by extending many of the results of Mori theory from threefolds to higher dimensions. An analytic proof in the case of varieties of general type was announced in \cite{Siu06}.

One of the main inputs of Mori theory is that instead of considering varieties, we should consider {\em pairs\/} $(X,\Delta)$, where $X$ is a normal projective variety and $\Delta\geq0$ is a $\mathbb Q$-divisor on $X$ such that the {\em adjoint divisor\/} $K_X+\Delta$ is $\mathbb Q$-Cartier. In order to get a good theory one has to impose a further technical condition on $K_X+\Delta$, namely how its ramification formula behaves under birational maps. For our purposes, the condition means that $X$ is a smooth variety, the support of $\Delta$ has simple normal crossings, and the coefficients of $\Delta$ are strictly smaller than $1$, that is $\rfdown \Delta.=0$. Therefore, we might pose the following:
\begin{conjecture}\label{c_fg}
Let $X$ be a smooth projective variety. Let $B_1,\dots,B_k$ be $\mathbb{Q}$-divisors on $X$ such that $\rfdown B_i.=0$ for all $i$, and such that the support of $\sum_{i=1}^k B_i$ has simple normal crossings. Denote $D_i=K_X+B_i$ for every $i$.

Then the ring $R(X;D_1,\dots,D_k)$ is finitely generated.
\end{conjecture}
Divisorial rings of this form are called {\em adjoint rings}. This conjecture obviously generalises the affirmative answer to Question \ref{q_finite}, by choosing $k=1$ and $B_1=0$.

Here we survey a different approach to the finite generation problem, recently obtained in \cite{CL10a}, which avoids all the standard difficult operations of Mori theory. In that paper, we give a new proof of the following result, only using the Kawamata-Viehweg vanishing and induction on the dimension.
\begin{theorema}\label{t_cox}
Let $X$ be a smooth projective variety. Let $B_1,\dots,B_k$ be $\mathbb{Q}$-divisors on $X$ such that $\rfdown B_i.=0$ for all $i$, and such that the support of $\sum_{i=1}^k B_i$ has simple normal crossings. Let $A$ be an ample $\mathbb{Q}$-divisor on $X$, and denote $D_i=K_X+A+B_i$ for every $i$.

Then the ring $R(X;D_1,\dots,D_k)$ is finitely generated.
\end{theorema}
Theorem \ref{t_cox}, in fact, also gives the affirmative answer to Question \ref{q_finite} -- details can be found in \cite{CL10a}.
This theorem was originally proved in \cite{BCHM10} as a consequence of Mori theory, and later in \cite{Laz09} by induction on the dimension and without Mori theory; for an easy introduction to the latter circle of ideas, see \cite{Corti11}. In \cite{CL10a}, we build on and significantly simplify arguments from \cite{Laz09}.

The aim of this survey is to show that, in spite of its length and modulo some standard technical difficulties, our paper \cite{CL10a} is based on simple ideas. We hope here to make these ideas more transparent and the general flow of the proof more accessible.

Moreover, by the results of \cite{CL10}, Theorem \ref{t_cox} implies all currently known results of Mori theory, so in effect this approach inverts the conventional order of the program on its head, where earlier finite generation came at the end of the process, and not at the beginning. Further, it is shown that Conjecture \ref{c_fg} implies the most of outstanding conjectures of the theory, in particular the Abundance conjecture.

Finally, one might ask whether the rings $R(X;D_1,\dots,D_k)$ are finitely generated when the divisors $D_i$ are not adjoint, or whether the assumption on singularities can be relaxed. However, Example \ref{e_cutkosky} shows that this ring might not be finitely generated in similar situations.\\

\paragraph{\bf Notation and Conventions}
We work with algebraic varieties  defined over
$\mathbb C$. We denote by $\mathbb R_+$ and $\mathbb Q_+$ the sets of non-negative
real and rational numbers, and by $\overline{\mathcal C}$ the topological closure of a set $\mathcal C\subset\mathbb R^N$.

Let $X$ be a smooth projective variety and $\mathbf R\in \{\mathbb Z,\mathbb Q,\mathbb R\}$. We
denote by $\Div_{\mathbf R}(X)$ the group of
$\mathbf R$-divisors on $X$, and $\sim_{\mathbf R}$ and $\equiv$ denote the $\mathbf R$-linear and numerical equivalence of $\mathbb R$-divisors.
If $A=\sum a_iC_i$ and $B=\sum b_iC_i$ are two
$\mathbb{R}$-divisors on $X$, then $\rfdown A.$ is the round-down of $A$, $\rcup A.$ is the round-up of $A$, and
$$
A \wedge B= \sum \min\{a_i,b_i\} C_i.
$$
Further, if $S$ is a prime divisor on $X$, $\mult_S A$ is the order of vanishing of $A$ at the generic point of $S$.

In this paper, a {\em log pair} $(X, \Delta)$ consist of a smooth variety $X$ and an $\mathbb R$-divisor $\Delta
\ge 0$. A projective birational morphism $f\colon\map Y.X.$ is a {\em log resolution} of the
pair $(X, \Delta)$ if $Y$ is smooth, $\Exc f$ is a divisor and the support of
$f_*^{-1}\Delta+\Exc f$ has simple normal crossings.
A log pair $(X,\Delta)$ with $\rfdown \Delta.=0$ has {\em canonical\/} singularities if for every log resolution $f\colon \map Y.X.$, we have
$$K_Y+f^{-1}_*\Delta=f^*(K_X+\Delta)+E$$
for some $f$-exceptional divisor $E\geq0$.

If $X$ is a smooth projective variety and $D$ is an integral divisor, $\bs|D|$ denotes the base locus of $D$, and $\fix|D|$ and $\mob(D)$ denote the {\em fixed\/}
and {\em mobile\/} parts of $D$. Hence
$|D| = |\mob(D)| + \fix|D|$, and the base locus of $|\mob(D)|$ contains no divisors. More
generally, if $V$ is any linear system on $X$, $\fix(V )$ is the fixed divisor of $V$.
If $S$ is a prime divisor on $X$ such that $S\nsubseteq\fix|D|$, then $|D|_S$ denotes the image of the linear system $|D|$ under restriction
to $S$.

If $D$ is an $\mathbb R$-divisor on $X$, we denote by $\Bs(D)$ the intersection of the sets $\Supp D'$ for all $D'\geq0$ such that $D'\sim_\mathbb R D$, and we call $\Bs(D)$ the {\em stable base locus} of $D$; we set $\Bs(D)=X$ if no such $D'$ exists.

\begin{definition}\label{d_variouspolytopes}
Let $X$ be a smooth variety, and let $S,S_1,\dots,S_p$ be distinct prime divisors such that $S+\sum_{i=1}^p S_i$ has simple normal crossings. Let $V=\sum_{i=1}^p\mathbb R S_i\subseteq \Div_{\mathbb R}(X)$, and let $A$ be a  $\mathbb Q$-divisor on $X$.
We define
\begin{align*}
\mathcal L(V)&=\{ B=\textstyle\sum b_iS_i\in V \mid 0\le b_i\le 1\text{ for
all }i\},\\
\mathcal E_A(V)&=\{B\in\mathcal L(V)\mid \text{there exists }0\leq D\sim_{\mathbb R}K_X+A+B\},\\
\mathcal B_A^S(V)&=\{B\in\mathcal L(V)\mid
S\nsubseteq\Bs(K_X+S+A+B)\}.
\end{align*}
\end{definition}

Now, let $X$ be a smooth projective variety, let $V\subseteq \Div_{\mathbb R}(X)$ be a finite dimensional vector space, and let $\mathcal C\subseteq V$ be a {\em rational polyhedral cone\/}, i.e.\ a convex cone spanned by finitely many rational vectors.
Then the {\em divisorial ring associated to $\mathcal C$\/} is
$$R(X,\mathcal C)=\bigoplus_{D\in \mathcal C\cap \Div_{\mathbb Z}(X) }H^0(X,\ring X. (D)).$$
Note that $\mathcal{C}\cap \Div_{\mathbb Z}(X)$ is a finitely generated monoid by Gordan's lemma, and if all elements in it are multiples of adjoint divisors, the corresponding ring is an {\em adjoint ring\/}. This generalises divisorial and adjoint rings introduced earlier.

If $S$ is a prime divisor on $X$, the {\em restriction of $R(X,\mathcal C)$ to
$S$\/} is defined as
$$\rest_S R(X,\mathcal C)=\bigoplus_{D\in\mathcal{C}\cap \Div_{\mathbb Z}(X)}\rest_S
H^0\big(X,\ring X.(D)\big),$$
where $\rest_S H^0\big(X,\ring X.(D)\big)$ is the image of the restriction map
$$H^0(X,\ring X.(D))\longrightarrow H^0(S,\ring S.(D)).$$
Similarly, we denote by $\rest_S R(X;D_1,\dots,D_k)$ the restricted divisorial ring for $\mathbb Q$-divisors $D_1,\dots,D_k$ in $X$.

We finish this section with some basic definitions from convex geometry. Let $\mathcal{C}\subseteq \mathbb R^N$ be a convex set.
A subset $F\subseteq\mathcal C$ is a \emph{face} of $\mathcal{C}$ if $F$ is
convex, and
whenever $u+v\in F$ for $u,v\in\mathcal C$, then $u,v\in F$.  Note that $\mathcal C$ is itself a face of $\mathcal C$.
We say that $x\in \mathcal{C}$ is an \emph{extreme point} of $\mathcal C$ if $\{x\}$ is a face of $\mathcal{C}$.
It is a well known fact that any compact convex set $\mathcal C\subseteq \mathbb R^N$ is the convex hull of its extreme points.

A \emph{polytope} in $\mathbb R^N$ is a compact set which is
the intersection of finitely many half spaces; equivalently, it is
the convex hull of finitely many points in $\mathbb R^N$. A polytope is
\emph{rational} if it is
an intersection of rational half spaces; equivalently, it is
the convex hull of finitely many rational points in $\mathbb R^N$.

\section{Zariski decomposition on surfaces}

Zariski decomposition, introduced by Zariski in \cite{Zar62} in the case of surfaces, has proven to be a powerful tool in birational geometry.
In this section, after reviewing some basic notions about the Zariski decomposition and its generalization in higher dimension,  we show that the finite generation of the canonical ring on surfaces is an easy consequence of this decomposition, combined with the Kawamata-Viehweg vanishing theorem.
It is worth noting that while the vanishing theorem holds in any dimension, and it is the most important ingredient for the proof of the lifting theorem described in the next section, the Zariski decomposition a priori only holds on surfaces, see \cite[\S 2.3.E]{Lazarsfeld04a}.

\begin{definition}
Let $X$ be a smooth projective variety and let $D$ be a pseudo-effective $\mathbb Q$-divisor. Then $D$ admits a {\em Zariski decomposition} if there exist a nef $\mathbb Q$-divisor $P$ (positive part) and a $\mathbb Q$-divisor $N\ge 0$ (negative part) such that
\begin{enumerate}
 \item $D=P+N$,
\item for any positive integer $m$, the natural homomorphism
$$H^0(X,\mathcal O_X(\lfloor mP\rfloor))\longrightarrow H^0(X,\mathcal O_X(\lfloor mD\rfloor))$$
is an isomorphism.
\end{enumerate}
In particular, we have $R(X,D)\simeq R(X,P)$.
\end{definition}

In the case of surfaces, Zariski showed

\begin{theorem}\label{t_zariski}
Let $X$ be a smooth projective surface and let $D$ be a pseudo-effective $\mathbb Q$-divisor.
Then there exist unique $\mathbb Q$-divisors
$P$  and $N=\sum_{i=1}^r\nu_i N_i\geq0$ such that:
\begin{enumerate}
 \item  $D=P+N$,
\item $P$ is nef,
\item $P\cdot N_i=0$ for every $i$,
\item the intersection matrix $(N_i\cdot N_j)_{i,j}$ is negative definite.
\end{enumerate}
In particular, this defines a Zariski decomposition of $D$.
\end{theorem}

Unfortunately, in higher dimensions it is easy to find examples of divisors which do not admit a Zariski decomposition, even after taking their pull-back to a suitable modification, see \cite{Nakayama04}. This birational version of the decomposition is usually called Cutkosky-Kawamata-Moriwaki decomposition.

However, if $D$ is a big $\mathbb Q$-divisor on a smooth surface $X$ and if $N=\sum\nu_i N_i$ is the negative part of $D$ as in Theorem \ref{t_zariski}, it can be shown that
$$\nu_i=\limsup\limits_{m\rightarrow\infty}\frac1m \mult_{N_i}|mD|,$$
see \cite[Corollary 2.3.25]{Lazarsfeld04a}. This motivates the following definition
 of the positive and negative parts of a pseudo-effective divisor over a smooth projective variety. These were introduced by Nakayama in \cite{Nakayama04} (see also \cite{Boucksom04} for an analytic construction).

\begin{definition}
Let $X$ be a smooth projective variety, let $A$ be an ample $\mathbb{R}$-divisor,
and let $\Gamma$ be a prime divisor. If $D\in\Div_\mathbb{R}(X)$ is a big divisor, define
$$
\sigma'_\Gamma(D)=\inf\{\mult_\Gamma D'\mid 0\leq D'\sim_{\mathbb R} D\}.
$$
If $D\in\Div_\mathbb{R}(X)$ is pseudo-effective, set
$$
\sigma_\Gamma(D)=\lim_{\varepsilon\to 0} \sigma'_\Gamma(D+\varepsilon A).$$
Then the {\em negative part\/} of $D$ is defined as $$\textstyle N_\sigma(D)=\sum_\Gamma\sigma_\Gamma(D)\cdot\Gamma,
$$
where the sum runs over all prime divisors $\Gamma$ on $X$.
The {\em positive part\/} of $D$ is  given by $P_\sigma (D)=D-N_\sigma(D)$.
\end{definition}

Then one easily shows that $\sigma'_\Gamma(D)=\sigma_\Gamma(D)$ when $D$ is big. Note that, unlike in the case of surfaces,
$P_\sigma(D)$ is in general not nef. The main reason for this is that on surfaces, every mobile divisor is nef since curves and divisors coincide. Nevertheless, this decomposition turned out to be very useful in the recent advances of the Minimal Model Program, e.g.\ see \cite{BCHM10}. The following result summarises its main properties, see \cite{Nakayama04}.

\begin{lemma}\label{d_sigma}
Let $X$ be a smooth projective variety, and let $D$ be a pseudo-effective $\mathbb R$-divisor.

Then $N_\sigma(D)$ depends only on the numerical class of $D$. Moreover, if $\Gamma$ is a prime divisor, then
the function $\sigma_\Gamma$ is homogeneous of degree one, convex and lower semi-continuous on the cone of
pseudo-effective divisors on $X$, and continuous on the cone of big divisors.
\end{lemma}

We now describe the role of Zariski decompositions in
the study of the canonical ring of a smooth projective variety.
By the Basepoint free theorem \cite[Theorem 3.3]{km98}, it follows immediately that if $X$ is a smooth minimal projective variety of general type, i.e. if $K_X$ is big and nef, then $K_X$ is semi-ample, and thus the canonical ring $R(X,K_X)$ is finitely generated.
In  \cite{Kawamata87}, Kawamata generalized this result by showing that if $X$ is a smooth projective variety of general type, the existence of a Zariski decomposition for $K_X$ implies the finite generation of the canonical ring. Using the same techniques, we next show how Theorem \ref{t_zariski} implies the finite generation of the canonical ring on surfaces.

We first recall the following important  result of Zariski \cite{Zar62} on the structure of base loci of linear systems.

\begin{theorem}\label{t_baselocus}
Let $X$ be a projective variety, and let $D$ be a Cartier divisor on $X$ such that the base locus of $D$ is a finite set.

Then $D$ is semiample.
\end{theorem}

\begin{corollary}\label{c_mobilefix}
Let $X$ be a smooth projective surface and let $D$ be a $\mathbb Q$-divisor on $X$ such that $\kappa(X,D)\ge 0$.

Then, for every sufficiently divisible positive integer $k$, there exist a semiample $\mathbb Q$-divisor $M_k$
whose coefficients are at most $1/k$, and a divisor $F_k\geq0$ such that $\Supp F_k\subseteq \Bs(D)$ and $D\sim_\mathbb{Q}M_k+F_k$. Furthermore, if the support of $D$ has simple normal crossings, then the support of $M_k+F_k$ has simple normal crossings.
\end{corollary}
\begin{proof}
Let $m$ be a positive integer such that $\Bs(D)=\bs|mD|$. If we write
$$mD=M+F,$$
 where $M$ is the mobile part of $|mD|$ and $F$ is its fixed part, then $M$ is semiample by Theorem \ref{t_baselocus} and $\Supp F\subseteq \Bs(D)$. Thus, for every sufficiently divisible positive integer $k$,
the linear system $|kM|$ is basepoint free. Bertini's theorem implies that a general section $N_k$ of $|kM|$ is reduced and irreducible. Thus,  it is enough to define $M_k=N_k/km$ and $F_k=F/m$. The last claim follows from the construction.
\end{proof}

Our goal is to show that for any smooth projective surface $X$ and for any divisor $\Delta$ on $X$ such that the support of $\Delta$ has simple normal crossings and $\rfdown\Delta.=0$, the ring $R(X,K_X+\Delta)$ is finitely generated.
As in the higher dimensional case, see \cite{BCHM10} and \cite{CL10a}, by the existence of log resolutions and by a result of Fujino and Mori \cite{FM00}, we may and will assume, without loss of generality, that the
divisor $\Delta$ can  be written as $\Delta=A+B$, where $A$ is an ample $\mathbb Q$-divisor and $B$ is a $\mathbb Q$-divisor such that the support of $B$ has simple normal crossings and $\rfdown B.=0$.

First we need an easy result about adjoint divisors on curves.

\begin{lemma}\label{l_adjointcurves}
Let $X$ be a smooth projective curve of genus $g\geq1$, let $D$ be an integral divisor, and let $\Theta$ be a $\mathbb Q$-divisor such that such that $\deg\Theta>0$ and
$D\sim_\mathbb{Q} K_X+\Theta$.

Then $H^0(X,\ring X. (D))\neq0$.
\end{lemma}
\begin{proof}
By Riemann-Roch theorem we have
$$h^0(X,\ring X. (D))\geq\deg D-g+1=2g-2+\deg\Theta-g+1>g-1\geq0,$$
which proves the lemma.
\end{proof}

We recall the following particular version of the Kawamata-Viehweg vanishing theorem which we need in this paper.

\begin{theorem}
Let $X$ be a smooth variety and let $B=\sum b_jB_j$ be a $\mathbb Q$-divisor whose support has simple normal crossings and such that $0\leq b_j\leq1$ for each $j$. Let $A$ be an ample $\mathbb Q$-divisor and assume $D$ is an integral divisor such that $D\sim_\mathbb Q K_X+B+A$.

Then $H^i(X,\mathcal O_X(D))=0$ for every $i>0$.
\end{theorem}

Finally, we have:

\begin{theorem}
Let $X$ be a smooth projective surface, and let $B$ be a $\mathbb Q$-divisor such that the support of $B$ has simple normal crossings and $\rfdown B.=0$.
Let $A$ be an ample $\mathbb Q$-divisor and let $\Delta=A+B$.

Then the ring $R(X,K_X+\Delta)$ is finitely generated.
\end{theorem}

\begin{proof}
We may assume that $\kappa(X,K_X+\Delta)\geq0$, since otherwise $R(X,K_X+\Delta)$ is trivial.
Let $K_X+\Delta=P+N$ be the Zariski decomposition of $K_X+\Delta$, whose existence is guaranteed by Theorem \ref{t_zariski}. Since $R(X,K_X+\Delta)\simeq R(X,P)$, we may assume that $P$ is not semiample, and in particular, $\Bs(P)$ contains a curve by Theorem \ref{t_baselocus}. Furthermore, we can assume that $P\geq0$.

Replacing $X$ by a log resolution, we may assume that the support of $B+P+N$ has simple normal crossings. For every positive integer $k$, let $P\sim_\mathbb Q M_k+F_k=P_k$ be the decomposition as in Corollary \ref{c_mobilefix}.
Note that $\rfdown B-N.\leq0$, and let
$$\lambda_k=\sup \{t\geq0\mid \rfdown B+tP_k-N.\le 0 \}.$$
Then $\lambda_k>0$, and if $\Sigma_k$ is the sum of all the prime divisors in  $B+\lambda_k P_k - N$ with coefficient equal to $1$, then the support of $\Sigma_k$ is contained in the support of $P_k$. Furthermore, by choosing $k$ large enough, we can assume that the support of $\Sigma_k$ is contained in $\Bs(P)$, and hence without loss of generality, we may assume that $P=P_k$. Denote $\Sigma=\Sigma_k$ and $\lambda=\lambda_k$.

Let
$$R=\sum _T (\mult_T \rcup N-B-\lambda P. )~ T,$$
where the sum is over all prime divisors $T$ such that $\mult_T (N-B-\lambda P) >0$, and denote $B'=B+\lambda P - N+R$. Then $R$ is an integral divisor, the coefficients of $B'$ lie in the interval $(0,1]$, and we have
$$0\le R \le \rcup N., \qquad\text{and} \qquad\rfdown B'.=\Sigma.$$

Fix a prime divisor $S$ in $\Supp\Sigma\subseteq\Bs(P)$.
Let $m>\lambda+1$ be a sufficiently large positive integer such that $mA$, $mB$, $mP$ and $mN$ are integral divisors and $\Bs(P)=\bs|mP|$.
Let $A'=A+(m-1 - \lambda)P$, and note that $A'$ is ample since $P$ is nef and $m-1 - \lambda>0$.
Then
$$m P + R\sim_\mathbb{Q}K_X+A+B-N+(m-1)P+R=K_X+A'+B',$$
and the Kawamata-Viehweg vanishing theorem implies $H^1(X,\mathcal O_X(mP+R-S))=0$.

We claim that
$$H^0(S,\mathcal O_S(mP+R))\neq0.$$
We first show that the claim implies the Theorem. The long cohomology sequence associated to the exact sequence
$$0\longrightarrow\mathcal O_X(mP+R-S)\longrightarrow\mathcal O_X(mP+R)\longrightarrow\mathcal O_S(mP+R)\longrightarrow0$$
yields that the map $H^0(X,\mathcal O_X(mP+R))\longrightarrow H^0(S,\mathcal O_S(mP+R))\neq0$ is surjective, and therefore $S\nsubseteq\bs|mP+R|$.
Since $0\leq R\leq\rcup N.\leq\rcup mN.=mN$, the composition of injective maps
$$H^0(X,\mathcal O_X(mP))\longrightarrow H^0(X,\mathcal O_X(mP+R))\longrightarrow H^0(X,\mathcal O_X(mP+mN))$$
is an isomorphism by the property of Zariski decompositions, hence so is the first map. But this implies $\bs|mP+R|=\bs|mP|\cup\Supp R$, and thus $S\nsubseteq\bs|mP|=\Bs(P)$,
a contradiction.

Finally, to prove the claim, let $g$ be the genus of $S$. Note that $S\nsubseteq\Supp R$ by construction,
and therefore $\deg(mP+R)_{|S}\ge 0$ since $P$ is nef.
In particular, this implies the claim for $g=0$. If $g\geq1$, note that
$$(m P + R)_{|S}\sim_\mathbb{Q}K_S+A'_{|S}+(B'-S)_{|S}$$
and $\deg(A'_{|S}+(B'-S)_{|S})>0$, hence the claim follows by Lemma \ref{l_adjointcurves}.
\end{proof} 
\section{Diophantine approximation and a lifting theorem}

Both Diophantine approximation and the Kawamata-Viehweg vanishing theorem play crucial roles
in recent developments in Mori theory. Diophantine approximation was first introduced in birational geometry,
as an essential tool, by Shokurov in \cite{Shokurov03}, in order to give a conceptual proof of the existence of certain
surgery operations called {\em flips\/}.

The aim of this section is to sketch how these two tools can be naturally combined to
obtain many of the results in \cite{CL10a}.\\

\paragraph{\bf Diophantine approximation}

We first recall Diophantine approximation, see for instance \cite[Lemma 3.7.7]{BCHM10}.

\begin{lemma}\label{l_diophant}
Let $\|\cdot\|$ be a norm on $\mathbb R^N$ and let $x\in \mathbb R^N$.
Fix a positive integer $k$ and a positive real number $\varepsilon$.

Then there are finitely many
points $x_i\in \mathbb R^N$ and positive integers $k_i$ divisible by $k$, such
that $k_i
x_i/k$ are integral, $\|x-x_i\|<\varepsilon/k_i$, and $x$
is a convex linear combination of $x_i$.
\end{lemma}

The previous Lemma implies the following characterization of rational polytopes in $\mathbb R^N$ (similar techniques
were used to prove  \cite[Theorem 4.4]{CL10a}).

\begin{proposition}\label{l_polytope}
Let $\|\cdot\|$ be a norm on $\mathbb R^N$ and  let $\mathcal P\subseteq \mathbb
R^N$ be a bounded convex subset.

Then $\mathcal{P}$ is a rational polytope if and only if there is a constant
$\varepsilon>0$ and a positive integer $k$ with the following property:
for every rational $v\in \mathbb R^N$, if there exist
$w\in\mathcal P$ and a positive integer $l$ such that $lv$ is integral and
$\|v-w\|<\varepsilon/lk$, then $v\in \mathcal{P}$.
\end{proposition}
\begin{proof}
Since any two norms on $\mathbb R^N$ are equivalent, we can assume that $\|\cdot\|$
is the standard Euclidean norm. Let $\langle\cdot\,,\cdot\rangle$ be the standard scalar product.

Suppose that $\mathcal{P}$ is a rational polytope. Then there exist finitely many
$c_i\in\mathbb Z$  and $\psi_i\in \mathbb Z^N$ such that
$w\in \mathcal{P}$ if and only if $\langle \psi_i,w\rangle\geq c_i$ for every
$i$. Pick $\varepsilon>0$ such that
$\|\psi_i\|<1/\varepsilon$ for all $i$ and let $k$ be any positive integer.

Assume that for $v\in \mathbb R^N$ there exist $w\in\mathcal{P}$
and a positive integer $l$ such that $lv$ is integral and $\|v-w\|<\varepsilon/lk$.
Then by the Cauchy-Schwarz inequality we have
$$
c_i-\langle \psi_i,v\rangle \le \langle\psi_i, w\rangle-\langle \psi_i,v\rangle=
\langle\psi_i, w-v\rangle\leq\|\psi_i\|\|w-v\|< \|\psi_i\|
\varepsilon/lk < 1/lk
$$
for any $i$. But $lk(c_i-\langle \psi_i,v\rangle)=lkc_i-\langle \psi_i,klv\rangle$ is an
integer, and so $\langle \psi_i,v\rangle\geq c_i$.  Therefore
$v\in \mathcal{P}$.

Assume now that there exists $\varepsilon$ and $k$ as in the statement of the proposition, and let  $w\in
\overline{\mathcal{P}}$.  By Lemma \ref{l_diophant}, there
exist points $w_i\in \mathbb R^N$ and positive integers $m_i$
divisible by $k$, such that $w$ is a convex linear combination of $w_i$ and
$$
\|w-w_i\|<\varepsilon/m_i\qquad\text{and}\qquad m_iw_i/k \text{ is integral}
$$
for every $i$. In particular, there exists $w'\in \mathcal P$ such that
$\|w'-w_i\|<\varepsilon/m_i$.
By the assumption, it follows that $w_i\in \mathcal{P}$ for all $i$, and in particular $w\in \mathcal P$.
Therefore, $\mathcal P$ is a closed set, and moreover, every extreme point of $\mathcal{P}$ is rational.

If $\mathcal P$ is not a rational polytope then there  exist infinitely many extreme points $v_i$ of
$\mathcal{P}$, with $i\in \mathbb N$. Since $\mathcal{P}$ is compact,
by passing to a subsequence
we obtain that there exist $v_\infty \in \mathcal P$
such that
$$v_\infty=\lim\limits_{i\rightarrow\infty}
v_i.$$
By Lemma \ref{l_diophant}, there exists a
positive integer $m$ divisible by $k$ and
$v'_\infty \in \mathbb R^N$ such that $mv'_\infty /k$ is integral and
$\|v_\infty-v'_\infty\|<\varepsilon/m$. By assumption, it follows that
$v'_\infty\in \mathcal P$. Pick $j\gg0$ so that
$$\|v_j-v'_\infty\|\le \|v_j-v_\infty\| + \|v_\infty-v'_\infty\|
<\frac \varepsilon m.$$
Therefore, there exists a positive integer $m'\gg0$ divisible by $k$ such that
$(m+m')v_j/k$ is integral, and such that if we define
$$v_0=\frac {m+m'}{m'}v_j-\frac m {m'}v'_\infty\in v'_\infty+\mathbb
R_+(v_j-v'_\infty),$$
then $m'v_0/k$ is
integral and
$$\|v_0-v_j\|=\frac m {m'}\|v_j-v'_\infty\|< \frac \varepsilon
{m'}.$$
By assumption, this implies that $v_0\in \mathcal{P}$, and since
$v_j=\frac{m'}{m+m'}v_0+\frac{m}{m+m'}v'_\infty$, it follows
that $v_j$ is not an extreme point of $\mathcal P$, a contradiction.

Thus, $\mathcal{P}$ is a rational polytope.
\end{proof}

\paragraph{\bf Lifting property}

Let $X$ be a smooth projective variety, let $S$ be a prime divisor, let $A$ be an ample $\mathbb Q$-divisor, and let $B\geq0$ be a $\mathbb Q$-divisor such that $S\nsubseteq\Supp B$, $\lfloor B\rfloor=0$, and $\Supp(S+B)$ has simple normal crossings.
Let $C\ge 0$ be a $\mathbb Q$-divisor on $S$ such that $C\le B_{|S}$, and let $m$ be a positive integer such that $mA$, $mB$ and $mC$ are integral.

We say that $(B,C)$ satisfies the {\em lifting property} $\mathfrak L_m$ if the image of the restriction morphism
$$H^0\big(X,\mathcal O_X(m(K_X+S+A+B))\big)\longrightarrow H^0\big(S,\mathcal O_S(m(K_S+A_{|S}+B_{|S}))\big)$$
contains $H^0\big(S,\mathcal O_S(m(K_S+A_{|S}+C))\big)\cdot\sigma$, where $\sigma$ is a global section of $\mathcal O_S(m(B_{|S}-C))$ vanishing along $m(B_{|S}-C)$.\\

The following theorem is \cite[Theorem 3.4]{CL10a}, and it is a slight generalization of the lifting  theorem by Hacon-M\textsuperscript cKernan \cite{HM10}, which is itself a generalization of  results by Siu \cite{Siu98,Siu02} and  Kawamata \cite{Kawamata99}. Similar results were also obtained in \cite{Takayama06}, \cite{Paun05} and
\cite{EP09}.

\begin{theorem}\label{t_lifting}
Let $X$ be a smooth projective variety, let $S$ be a prime divisor, let $A$ be an ample $\mathbb Q$-divisor, and let $B\geq0$ be a $\mathbb Q$-divisor such that $S\nsubseteq\Supp B$, $\lfloor B\rfloor=0$, and $\Supp(S+B)$ has simple normal crossings.
Let $C\ge 0$ be a $\mathbb Q$-divisor on $S$ such that $(S,C)$ is canonical, and let $m$ be a
positive integer such that
$mA$, $mB$ and $mC$ are integral.

Assume that there exists a positive integer $q\gg 0$ such
that $qA$ is very ample,
$S\not\subseteq\base|qm(K_X+S+A+B+\frac1m A)|$ and
$$C \le B_{|S}- B_{|S}\wedge \frac
1 {qm} \fix |qm(K_X+S+A+B+\textstyle\frac1m A)|_S.$$
Then $(B,C)$ satisfies $\mathfrak{L}_m$.
\end{theorem}

We omit the proof of the Theorem, but we emphasise that it is a direct consequence of the Kawamata-Viehweg vanishing plus some elementary arithmetic.

Before we show how the lifting theorem and Diophantine approximation are related to each other, it is useful to spend a few words on the assumptions of the theorem. It is well known that, in general, the lifting theorem does not hold if
we just take $C=B_{|S}$.  A simple counterexample is given by considering the blow-up of $\mathbb P^2$ at one point, see \cite{H05,CKL11} for more details. On the other hand, the condition that $(S,C)$ is canonical is a bit more subtle and it might look artificial, but Example \ref{e_canonical} shows that it is essential. Clearly if $X$ is a surface, this condition is guaranteed by the fact that $X$ is smooth, $\Supp(S+B)$ has simple normal crossings and $\rfdown B.=0$.

We now proceed with

\begin{theorem}\label{t_rational_polytope}
Let $X$ be a smooth projective variety, and let $S,S_1,\dots,S_p$ be distinct prime divisors such that $S+\sum_{i=1}^p S_i$ has simple normal crossings. Let $V\subseteq \Div_{\mathbb R}(X)$ be the subspace spanned by $S_1, \dots, S_p$, and let $W\subseteq \Div_{\mathbb R}(S)$ be the subspace spanned by all the components of $S_{i|S}$. Let  $A$ be an ample $\mathbb Q$-divisor on $X$.
Let $\mathcal Q'$ be the set of  pairs  of rational divisors
$$(B,C) \in \mathcal B_A^S(V)\times \mathcal E_{A_{|S}}(W)$$
such that
 $$ C\le B \quad \text{and} \quad (B,C) \text{ satisfies }\mathfrak L_m \text{ for infinitely many } m,$$
and let $\mathcal Q$ be the intersection of the closure of $\mathcal Q'$ and the convex hull of $\mathcal Q'$. 

Then $\mathcal Q$ is a finite union of rational polytopes.
\end{theorem}

Note that, since the aim is to provide a proof of Theorem \ref{t_cox} by induction on the dimension of $X$, in this context   we are assuming that Theorem \ref{t_cox} holds in dimension $\dim X-1$. In particular,  we may assume that $\mathcal E_{A_{|S}}(W)$ is a rational polytope, see \cite[Theorem 5.5]{CL10a}.

We now explain briefly how  Proposition \ref{l_polytope} can be applied to get Theorem \ref{t_rational_polytope}. First note that the property $\mathfrak L_m$ for the pair $(B,C)$ immediately implies that
$$\fix |m(K_S+(A+C)_{|S})|+m(B_{|S}-C) \ge \fix|m(K_X+S+A+B)|_S. $$

Thus, after doing some simple algebra of divisors,  the lifting result in Theorem \ref{t_lifting} can be rephrased by saying that if the property $\mathfrak L_m$ holds for a pair $(B',C')$ sufficiently ``close'' to $(B,C)$, then also the pair $(B,C)$ will satisfy the property $\mathfrak L_m$. Here the distance between $(B,C)$ and $(B',C')$ is bounded in terms of a positive integer $q$ such that $qB$ and $qC$ are integral.  Thus, we can apply Proposition \ref{l_polytope} to get the desired result.

Theorem \ref{t_rational_polytope} implies two crucial results which are related to Theorem \ref{t_cox}.

\begin{theorem}\label{t_consequences}
Under the assumption of Theorem \ref{t_rational_polytope}, the set $\mathcal B_A^S(V)$ is a rational polytope and, for any $B_1,\dots,B_k\in \mathcal L(V)$, the ring
$$\rest_S R(X;K_X+S+A+B_1,\dots,K_X+S+A+B_k)$$
is finitely generated.
\end{theorem}

To prove the first statement, it is sufficient to show that $\mathcal B_A^S(V)$ is the image of the set defined in Theorem \ref{t_rational_polytope} through the  first projection, and the result follows immediately by the convexity of $\mathcal B_A^S(V)$. Note that, in particular, $\mathcal B_A^S(V)$ is compact, which is one of the main ingredients in our proofs of several results in \cite{CL10a}. 

The second statement is more delicate. Theorem \ref{t_rational_polytope} implies that the restricted algebra is spanned by a finite union of adjoint rings on $S$. Thus, the result  follows  by induction on the dimension.
\section{Finite generation}

In this section we present the proof of a  special case of the finite generation theorem which already contains almost all fundamental problems of the general case, and it is
particularly easy to picture what is going on. We prove the following:

\begin{theorem}\label{t_fingen}
Let $X$ be a smooth projective variety, and let $S_1$ and $S_2$ be distinct prime divisors such that $S_1+S_2$ has simple normal crossings. Let $B=b_1S_1+b_2S_2$ be a $\mathbb Q$-divisor such that $0\le b_1,b_2<1$, and let $A$ be an ample $\mathbb Q$-divisor. Assume that $K_X+A+B\sim_\mathbb Q D$ for some $D\in \mathbb Q_+ S_1+ \mathbb Q_+ S_2$.

Then the ring $R(X,K_X+A+B)$ is finitely generated.
\end{theorem}

It will become clear from the scheme of the proof that it is necessary to work with higher rank algebras even in this simple situation. The proof will mostly be "by picture", and for that reason we restrict ourselves to the case of two components. The proof in the general case follows the same line of thought, the only difference is that it is more difficult to visualise. 

The following result will be used often without explicit mention in this section; the proof can be found in \cite{ADHL10}.

\begin{lemma}
Let $\mathcal{S}\subseteq\mathbb Z^r$ be a finitely generated monoid and let $R=\bigoplus_{s\in\mathcal S}R_s$ be an $\mathcal S$-graded algebra.
Let $\mathcal S'\subseteq\mathcal S$ be a finitely generated submonoid and let $R'=\bigoplus_{s\in\mathcal S'}R_s$.
\begin{enumerate}
\item If $R$ is finitely generated over $R_0$, then $R'$ is finitely generated over $R_0$.

\item If $R_0$ is Noetherian, $R'$ is a Veronese subring of finite index of $R$, and
$R'$ is finitely generated over $R_0$, then $R$ is finitely generated over $R_0$.
\end{enumerate}
\end{lemma}

\begin{proof}[Sketch of the proof of Theorem \ref{t_fingen}]
Let $V=\mathbb R S_1+\mathbb R S_2\subseteq\Div_{\mathbb{R}}(X)$ be the subspace spanned by $S_1$ and $S_2$,
let $\mathcal B\subseteq V$ be the rectangle with vertices $D,D+(1-b_1)S_1,D+(1-b_2)S_2$ and $D+(1-b_1)S_1+(1-b_2)S_2$,
and denote $\mathcal C=\mathbb R_+\mathcal B$.

For $i\in\{1,2\}$, consider the segments $$\mathcal B_i=[D+(1-b_i)S_i,D+(1-b_1)S_1+(1-b_2)S_2]\subseteq V$$ and the cones
$\mathcal C_i=\mathbb R_+\mathcal B_i$.
It is clear from the picture (see Figure \ref{pic1}) that $\mathcal C=\mathcal C_1\cup\mathcal C_2$, and that there exists $M>0$ such that the ``width'' of the cones $\mathcal C_i$
in the half-plane $\{xS_1+yS_2\mid x+y\geq M\}$ is bigger than $1$. More precisely,
\begin{enumerate}
\item if $xS_1+yS_2\in \mathcal{C}_i$ for $i\in\{1,2\}$ and for some $x,y\in\mathbb N$ with $x+y\ge M$, then $xS_1+yS_2-S_i\in \mathcal{C}$.
\end{enumerate}
\begin{figure}[htb]
\begin{center}
\includegraphics[width=0.62\textwidth]{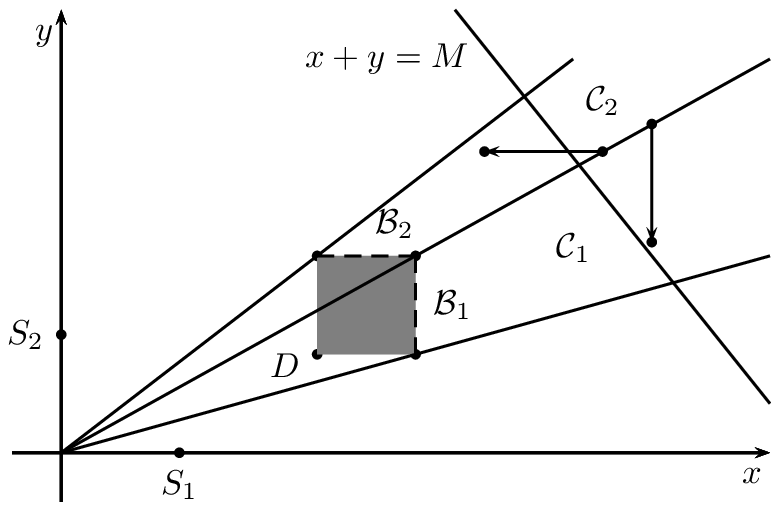}
\end{center}
\caption{}\label{pic1}
\end{figure}
We further claim that:
\begin{enumerate}
\item[(2)] for $i\in\{1,2\}$, the ring $\rest_{S_i} R(X,\mathcal C_i)$ is finitely generated.
\end{enumerate}
To show (2), without loss of generality we assume that  $i=1$. Let $\{D_1,\dots,D_\ell\}$ be a set of generators of $\mathcal C_1\cap\Div(X)$. Then
for every $j=1,\dots,\ell$, the line through $0$ and $D_j$ intersects the segment $\mathcal B_1$, and therefore, there exist rational numbers $0\le t_j\le 1-b_2$ and $k_j>0$ such that
$$D_j=k_j(D+ (1-b_1)S_1 + t_j S_2).$$
Since there is the natural projection
$$\rest_{S_1} R(X;D_1,\dots,D_\ell)\longrightarrow \rest_{S_1} R(X,\mathcal C_1),$$
it suffices to show that the first ring is finitely generated, and hence that the ring
$$R_1=\rest_{S_1} R(X;D+(1-b_1)S_1+t_1S_2,\dots,D+(1-b_1)S_1+t_\ell S_2)$$
is finitely generated. But this follows from Theorem \ref{t_consequences}, as
$$D+(1-b_1)S_1+t_jS_2\sim_\mathbb{Q} K_X+A+ S_1+(b_2+t_j) S_2.$$

Note that, in order to prove the theorem, it is enough to show that $R(X,\mathcal C)$ is finitely generated. Let $\sigma_i\in H^0(X,\ring X. (S_i))$  be
sections such that $\ddiv\sigma_i=S_i$, and let $\mathfrak R\subseteq R(X;S_1,S_2)$ be the ring spanned by $R(X,\mathcal C)$, $\sigma_1$ and $\sigma_2$.
Then it suffices to show that $\mathfrak R$ is finitely generated.

By (2), for $i\in\{1,2\}$ there are finite sets $\mathcal H_i$ of sections in the rings $R(X,\mathcal C_i)$ such that $\rest_{S_i} R(X,\mathcal C_i)$ are generated by the sets $\{\eta_{|S_i}\mid\eta\in\mathcal H_i\}$,
and denote $\mathcal H=\{\sigma_1,\sigma_2\}\cup\mathcal H_1\cup\mathcal H_2$. Let $\mathcal M$ be the intersection of the cone $\mathcal C$ with the half-plane $\{xS_1+yS_2\mid x+y\leq M\}$. Possibly by enlarging $\mathcal H$, we may assume that the elements of $\mathcal H$ generate $H^0(X,G)$ for every integral $G\in\mathcal M$.
\begin{figure}[htb]
\begin{center}
\includegraphics[width=0.62\textwidth]{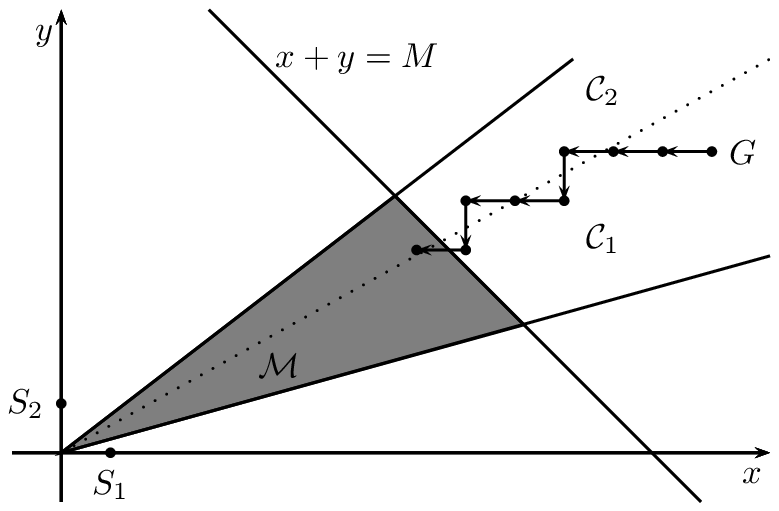}
\end{center}
\caption{}\label{pic2}
\end{figure}

We claim that $\mathcal H$ generates the whole ring $\mathfrak R$. Indeed, let $\chi \in \mathfrak R$. By definition of
$\mathfrak R$, we may write
$\chi=\sum_i\sigma_1^{\lambda_i}\sigma_2^{\mu_i}\chi_i$, where
$\chi_i\in H^0(X,\ring X. (G_i))$ for some integral $G_i\in \mathcal
C$ and some $\lambda_i,\mu_i\in \mathbb N$. Thus, it is enough to show that
$\chi_i$ are generated by the elements in $\mathcal H$, and therefore we may assume from the start that $\chi\in H^0(X,G)$ for some integral $G\in\mathcal C$. If $\chi\in\mathcal M$, we conclude by the definition of $\mathcal H$. Otherwise, assume that $G\in\mathcal C_1$, the case $G\in\mathcal C_2$ being analogous.
Then there are $\theta_1,\dots,\theta_z\in\mathcal H$ and a polynomial
$\varphi\in\C[X_1,\dots,X_z]$ such that $\chi_{|S_1}=\varphi(\theta_{1|S_1},\dots,\theta_{z|S_1})$, so
the exact sequence
$$0\longrightarrow H^0(X,\ring X.(G-S_1))\stackrel{\cdot \sigma_1}{\longrightarrow} H^0(X,\ring X.(G))\longrightarrow
H^0(S_1,\ring {S_1}.(G))$$
gives
$$\chi-\varphi(\theta_1,\dots,\theta_z)=\sigma_1\cdot\chi'$$
for some $\chi'\in H^0(X,\ring X. (G-S_1))$. Note that
$G-S_1\in\mathcal C$ by (1) above, and we continue with $\chi'$ instead of $\chi$.
This ``zig-zag'' process terminates after finitely many steps, once we reach $\mathcal M$ (see Figure \ref{pic2}).
We are done.
\end{proof}

\begin{remark}
In the case of surfaces, the ring $R(X,\mathcal C)$ which appears in the proof above, can be seen as the ring associated to the positive parts $P(D)$, for any divisor  $D\in\mathcal C$. By \cite{ELMNP}, the finite generation of this ring implies that the function $P(D)$ is piecewise linear on $\mathcal C$ and $P(D)$ is semiample for any  $\mathbb Q$-divisor $D\in \mathcal C$. This implies finiteness of ample models on   the rational polytope $\mathcal B$, see \cite{BCHM10}.  
\end{remark}

\section{Examples}

In this final section, we give examples which show that the results presented above are optimal.

\begin{example}\label{e_cutkosky}
This example is similar to \cite[Example 2.3.3]{Lazarsfeld04a}.  Here we show that in Theorem \ref{t_cox}, even for curves, the assumption of ampleness for the divisor $A$  cannot be replaced by nefness.

Let $X$ be an elliptic curve and let $A$ be a non-torsion integral divisor on $X$ of degree $0$. Let $B_1=0$ and $B_2\ge 0$ be a non-zero $\mathbb Q$-divisor such that $\rfdown B_2.=0$. Note that $A$ is nef and $(X,B_i)$ is canonical, for $i\in\{1,2\}$. We want to show that the ring $R=R(X;K_X+A+B_1,K_X+A+B_2)$ is not finitely generated.

To that end, let $k$ be a positive integer such that $kB_2$ is integral. We have that
$$R=\bigoplus_{(m_1,m_2)\in \mathbb N^2}R_{m_1,m_2},$$
where $R_{m_1,m_2}=H^0(X,\ring X. (\rfdown(m_1+m_2)A+m_2B_2.))$. If $R$ were finitely generated, then the set $\mathcal S=\mathbb R_+\{(m_1,m_2)\in\mathbb N^2\mid R_{m_1,m_2}\neq0\}$ would be a rational polyhedral cone, and in particular a closed subset of $\mathbb R^2$. However, we have $R_{m_1,0}=0$ for all $m_1>0$, and $R_{m_1,k}\neq0$ for every $k>0$ by Riemann-Roch, since $(m_1+k)A+kB_2$ has positive degree. Therefore $\mathcal S=\mathbb R_+^2\backslash\{(r,0)\mid r>0\}$, hence $R$ is not finitely generated.
\end{example}

\begin{example}\label{e_canonical}
In this example, we show that in Theorem \ref{t_lifting}, the assumption that $(S,C)$ is canonical cannot be replaced by the weaker
assumption that only $\rfdown C.=0$. Below, we are allowed to take $C=B$. The construction is similar to Mukai's flop, see \cite{Totaro09} and \cite[1.36]{Debarre01}.

Let $\mathcal E=\ring \mathbb P^1.\oplus \ring \mathbb P^1. (1)^{\oplus 3}$ and let $X=\mathbb P(\mathcal E)$ with the projection map $\pi\colon\map \mathbb P(\mathcal E).\mathbb P^1.$. Thus, $X$ is a smooth projective $4$-fold. Let $S\simeq\mathbb P^3$ be a fibre of $\pi$ and denote $\xi=c_1(\ring X.(1))$. Then $\xi$ is basepoint free by \cite[Lemma 2.3.2]{Lazarsfeld04a}, and
$$K_X=\pi^*(K_{\mathbb P^1}+\det\mathcal E)-4\xi=S-4\xi.$$
The linear system $|\xi -S|$ contains smooth divisors $S_1,S_2,S_3$
corresponding to the quotients $\mathcal E \longrightarrow \ring \mathbb P^1.\oplus \ring \mathbb P^1. (1)^{\oplus 2}$, and it is obvious that $S+S_1+S_2+S_3$ has simple normal crossings. If we denote
$$B= \frac 8 9 (S_1+S_2+S_3)  \qquad \text{and}\qquad A= \frac 7 3 \xi + \frac 1 6 S,$$
then $A$ is ample and $B\sim_{\mathbb Q}\frac 8 3( \xi -S)$. Note that $\rfdown B.=0$, and it is easy to check that $(S,B_{|S})$ is not canonical. Setting $\Delta=S+A+B$, we have
$$K_X+\Delta\sim_{\mathbb Q}\xi -\frac 1 2 S.$$
Let $P$ be the curve in $X$ corresponding to the trivial quotient of $\mathcal E$. Then $P=S_1\cap S_2\cap S_3$, and we have the relations $S\cdot P=1$ and $S\cdot \xi=0$. In particular, for any sufficiently divisible positive integer $m$, we have $P=\bs|m(\xi-S)|$ and
$$P\subseteq\bs|m(2\xi-S)|\subseteq\bs|m\xi|\cup\bs|m(\xi-S)|=P.$$
Therefore, the base locus of the linear system $|2m(K_X+\Delta)|$ is the curve $P$, and hence
$$\fix|2m(K_X+\Delta)|=0.$$
We want to show that for $m$ sufficiently divisible, the restriction map
$$H^0(X,\ring X. (2m(K_X+\Delta)))\to H^0(S,\ring S. (2m(K_S+(A+B)_{|S})))$$
is not surjective, thus demonstrating that the assumption of canonicity cannot be removed from Theorem \ref{t_lifting}. Assume the contrary. Then the sequence
\begin{align*}
0\longrightarrow H^0(X,\ring X. (2m(K_X+\Delta)-S))&\longrightarrow H^0(X,\ring X. (2m(K_X+\Delta)))\\
&\longrightarrow H^0(S,\ring S. (2m(K_S+(A+B)_{|S})))\longrightarrow 0
\end{align*}
is exact. After some calculations, we have
$$h^0(X,\ring X. (2m(K_X+\Delta)-S))=h^0(\mathbb P^1,S^{2m}\mathcal E(- m -1))=\sum_{j=m+1}^{2m}{j+2\choose2}(j-m)$$
and
$$h^0(X,\ring X. (2m(K_X+\Delta)))=h^0(\mathbb P^1,S^{2m}\mathcal E(- m))=\sum_{j=m}^{2m}{j+2\choose2}(j-m+1).$$
Furthermore, one sees that $K_S+(A+B)_{|S}$ represents a hyperplane in $S$, and thus
$$h^0(S,\ring S. (m(K_S+(A+B)_{|S})))=h^0(\mathbb P^3,\ring \mathbb P^3. (m))={m+3\choose m}.$$
It is now straightforward to derive a contradiction.
\end{example}

\bibliographystyle{amsalpha}
\bibliography{Library}

\end{document}